\documentclass[12pt,amsfonts]{article}
\usepackage[margin=1in]{geometry}  
\usepackage{graphicx,color}              
\usepackage{amsmath}               
\usepackage{amsfonts}              
\usepackage{amsthm} 

\usepackage[colorinlistoftodos]{todonotes}
\usepackage{tikz,enumitem}
\def\eps{\varepsilon}
\def\de{\partial}
\renewcommand{\t}{\theta}

\def\dis{\displaystyle}
\def\nd{\noindent}
\def\thend{\rule{3mm}{3mm}}
\def\R{\mathbb{R}}
\def\RD{\mathbb{R}^N}
\newcommand{\ird}{\int_{\mathbb{R}^N}}
\def\N{\mathbb{N}}

\newtheorem{thm}{Theorem}[section]
\newtheorem{prop}{Proposition}[section]
\newtheorem{lem}{Lemma}[section]

\newtheorem{deff}{Definition}[section]

\numberwithin{equation}{section}

\begin{document}

\title{
Solitary waves  for a class of generalized  Kadomtsev-Petviashvili equation  in  $\mathbb{R}^N$\\ with positive and zero mass} 

\author{\sf Claudianor O. Alves\thanks{Research of C. O. Alves partially supported by  CNPq/Brazil 304804/2017-7  and INCTMAT/CNPq/Brazil.} 
	\\
	\small{Universidade Federal de Campina Grande, }\\
	\small{Unidade Acad\^emica de Matem\'atica } \\
	\small{CEP: 58429-900 - Campina Grande-PB, Brazil}\\
	\small{ e-mail: coalves@mat.ufcg.edu}\\
	\vspace{1mm}\\
	{\sf Ol\'impio H. Miyagaki\thanks{ O.H.M. was partially supported by INCTMAT/CNPq/Brazil, CNPq/Brazil 304015/2014-8.
		}} \\
		{ \small Departamento de Matem\'atica}\\
		{\small Universidade Federal de  Juiz de Fora}\\
		{ \small CEP:  36036-330 - Juiz de Fora - MG, Brazil}\\
		{\small e-mail: ohmiyagaki@gmail.com}\vspace{3mm} \\		
		{\sf  Alessio Pomponio \thanks{ A. P. was  partially supported by  a grant of the group GNAMPA of INdAM and FRA2016 of Politecnico di Bari. }}\\
		{\small  Dipartimento di Meccanica, Matematica e Management}\\
		{\small  Politecnico di Bari}\\
		{\small  Via Orabona, 4, 70125 Bari,
			Italy}\\
		{\small  e-mail: alessio.pomponio@poliba.it}
	}
\date{}
\maketitle
\begin{abstract}
In this paper we use variational methods to establish existence of solitary waves  for a  class of   generalized Kadomtsev-Petviashvili $(GKP)$ equation in $\mathbb{R}^N$. The positive and zero mass  cases are considered. The main argument is to find a Palais Smale sequence satisfying a property related to Pohozaev identity, as in \cite{HIT}, which was used for the first time by \cite{jj}.
\end{abstract}

\vspace{0.5 cm}


\section{Introduction}

This paper deals with the  generalized Kadomtsev-Petviashvili (GKP) equation in  $\mathbb{R}^N$
\begin{equation}
\label{eq1*} u_t + u_{xxx}+(h(u))_x = D^{-1}_{x}\Delta_y u, 
\end{equation}
where $ (t,x,y)\times \R^{+}\times\R \times\R^{N-1},$ $ N \geq 2,$ $\displaystyle D^{-1}_{x} f(x,y)=\int_{-\infty}^{x}f(s,y)ds$ and $\displaystyle \Delta_y=\sum_{i=1}^{N-1}\frac{\partial^2}{\partial y^{2}_{i}}.$

We are interested, in particular, to the existence of a solitary wave for (\ref{eq1*}), that is, a solution $u$ of the form $u(t,x,y)=u(x-\tau t, y),$ with $\tau>0$. Hence, such function $u$ must satisfy the problem
\begin{equation*}\label{eq2*}
-\tau u_x + u_{xxx}+(h(u))_x = D^{-1}_{x}\Delta_y u, \qquad\hbox{in }\R^N,
\end{equation*}
or, equivalently,
\begin{equation}\label{eq4}
 (- u_{xx}+ \tau u + D^{-2}_{x}\Delta_y u -h(u))_x=0, \qquad\hbox{in }\R^N.
\end{equation}

We would like to point out that the equation (\ref{eq1*}) is a  two-dimensional  Korteweg-de Vries type equation when  $h(t)=t^{2}$, which  is a model for long dispersive waves, essentially unidimensional, but having small transverse effects, see  \cite{GKP}. For the   Cauchy problem associated with  equation  (\ref{eq1*}), we would like to cite, e.g. \cite{Bourgain, Fam, IsazaM} and the survey \cite{Saut}.  Generalized Kadomtsev-Petviashvili $(GKP)$ equation has been studied by many authors, see for instance,  G\"ung\"or  and Winternitz \cite{gungor}, Tian and Gao \cite{tian}, Zhang, Xu and Ma \cite{zhang} and references therein, where they focused in the study of solitary or soliton solutions, complete integrability, etc. 
The pioneering work is due to De Bouard and Saut in \cite{SautB2,SautB1}, where they treated a nonlinearity $h(s)=|s|^{p}s$ assuming that $p=\frac{m}{n}$, with $m$ and $n$ relatively prime, and $n$ is odd and $1 \leq  p < 4 ,$ if $N=2,$ or $1 \leq p< 4/3,$ if $N=3$.
In the mentioned paper, De Bouard and Saut obtained existence results for  equation (\ref{eq1*}) by combining minimization with concentration compactness  theorem \cite{Lions}. In \cite{SautB3}, instead, the authors proved that the solutions obtained in former papers are cylindrically symmetric. For the regularity of the solutions they assumed $ p=2,3,4$ if $N=2,$ and $p=2$ if $N=3.$ In an interesting paper \cite{Klein}, Klein and Saut used a  numerical simulation to analyse several quantitative properties of the De Bouard and Saut existence results, such as, blow-up, stability or instability of solitary waves. Also,  in this paper,  they study the zero-mass case and make a survey of various mathematical results on this subject. 
   In Willem \cite{Willem} and Wang and Willem \cite{WangWillem} the existence of a solitary waves for a class of $(GKP)$ problems of the type (\ref{eq1*}) were considered with an autonomous continuous nonlinearity $h$ with $N=2,$  and existence and multiplicity results have been proved, respectively.
 Their results were obtained by applying the mountain pass theorem \cite{AR} and Lusternik-Schnirelman theory, respectively.  In \cite{Liang}, Liang  and Su have proved the existence of solution for a class considered of the  $(GKP)$ problem  (\ref{eq1*}), which involves  a non autonomous  continuous function with $N \geq 2,$ while Xuan \cite{Xuan} treated the autonomous case in higher dimension. Applying the linking theorem stated in \cite{Szulkin}, He and Zou in \cite{HeZou} obtained nontrivial solution for  (\ref{eq1*}), in higher dimension without Ambrosetti-Rabinowitz growth condition given in \cite{AR} (see also \cite{Zou} and \cite{chineses} for multiplicity results).  We  recall that in all the  above papers, the regularity of the solitary waves have not been treated. Recently Alves and Miyagaki \cite{AlvesGKP}  have treated the case non autonomous, getting results similar to those obtained in, for instance, \cite{SautB2,SautB1}, and also the regularity properties of the solutions. Paumond in \cite{Paumond} obtained nonsymmetric solutions for  (\ref{eq1*}) with $N=5,$  extending those results in \cite{SautB2,SautB1}.

Motivated by above articles dealing with Kadomtsev-Petviashvili equation as well as by the fact that the variational methods can be employed to find a solitary waves for (\ref{eq1*}), this paper concerns with the existence of  solitary waves for problem \ref{eq4} with $\tau=1,$ in higher dimensions.

In whole this paper, the function $h$ belongs to $C^1(\R^N)$ and $h(0)=0$. Here, we are going to work with two classes of problems: 
the Positive Mass case and the Negative Mass case.
\\

\noindent {\bf Problem 1:} {\it Positive Mass}. In this case, we assume that $h$ satisfies the following conditions: 
\begin{enumerate}[label=($h_\arabic{*})$, ref=$h_\arabic{*}$]
\item \label{hp} for some $p\in (1, \bar{N}-1),$ where  $\bar{N}=\frac{2(2N-1)}{2N-3};$
\[
\lim_{|t| \to +\infty}\frac{h(t)}{|t|^p}=0;
\]
\item \label{h'}  $h'(0)=0$;

\item \label{h>0} there exists $\xi > 0$ such that $2H(\xi)-\xi^{2}>0$, where $H(t)=\int_{0}^{t}h(r)dr$.
\end{enumerate}

Our main result for this class of problems is the following

\begin{thm}\label{T1}  Suppose \eqref{hp}-\eqref{h>0}, then the problem (\ref{eq4}) possesses at least a nontrivial  solution. 
\end{thm}

 Our second class of problems is the following\\
 
 \noindent {\bf Problem 2:}\, {\it Zero Mass.} \\
 
 \noindent We suppose \eqref{h>0} and the conditions below
\begin{enumerate}[label=($h_\arabic{*})$, ref=$h_\arabic{*}$]\setcounter{enumi}{3}
	\item \label{hf}   $h(t)=f(t)+t$;
	\item\label{hf5} $\displaystyle \lim_{t \to 0}\frac{f(t)}{|t|^{\bar{N}-1}}=\lim_{|t| \to +\infty}\frac{f(t)}{|t|^{\bar{N}-1}}=0$.
\end{enumerate}
Observe that, by \eqref{h>0} and \eqref{hf}, defining $F(t)=\int_{0}^{t}f(r)dr$, we infer that 
\begin{equation}\label{f>0}
\hbox{there exists }\xi > 0\hbox{ such that }F(\xi)>0.
\end{equation}
For the Zero Mass case, the problem (\ref{eq4}) will be written of the form
\begin{equation}\label{eq4'}
(-u_{xx} - f(u) + D^{-2}_{x} \Delta u_{y})_x=0, \quad \mbox{in} \ \mathbb{R}^N.
\end{equation}
\\

Related with this class of problems we have the following result

\begin{thm}\label{T2}  Suppose \eqref{h>0}-\eqref{hf5}, then problem (\ref{eq4'}) possesses at least a nontrivial  solution. 
\end{thm}

One of the main motivations to study Problems 1 and 2 comes from the seminal papers due to Berestycki and Lions in \cite{BL}, for $N\geq 3,$ and  to Berestycki, Gallou\"et and Kavian in \cite{kavian}, for $N=2,$ where the authors proved the existence of a ground state, namely a solution which minimizes
the action among all the nontrivial solutions, for the problem
$$ \left\{ \begin{array}{l}
-\Delta u  = g(u), \quad \mbox{in} \ \R^N,\\
u \in H^1(\R^N), 
\end{array}\right. $$
under the following assumptions on the nonlinearity $g$:
\begin{description}
	\item[$	(g_1)$] $g$ is an odd and  continuous function;
	\item[$	(g_2)$] $-\infty < \liminf_{s \rightarrow 0^+}\frac{ g(s)}{s} \leq \limsup_{s \rightarrow 0^+} \frac{ g(s)}{s} = −m \leq 0;$
	\item[$	(g_3)$] $-\infty < \limsup_{s \rightarrow +\infty}\frac{ g(s)}{s^{2^*-1}}\leq  0,  \ (N \geq 3) , \limsup_{s\rightarrow +\infty}\frac{ g(s)}{e^{\alpha s^2}}\leq 0, \forall \alpha>0, \ (N=2);$
	\item[$	(g_4)$] there exists $\tau > 0 $ such that $G(\tau) := \int_{0}^{\tau}
	g(s) ds > 0;$
\end{description}	
where $2^* = 2N/(N-2).$
They considered two cases $m<0$ and $m=0$, so called \lq\lq positive mass\rq\rq and\lq\lq zero mass\rq\rq cases,
respectively. The  last case  is related to the Yang-Mills equation, see, e.g. \cite{Gidas}.  After these two pioneering papers, many researches worked in this subject, extending or improving in several ways, see, for instance, \cite{Monteo1, Monteo2, AP, Benci,  Micheletti, HIT,JT,AW}, and references therein: clearly the list is not complete.

It is very important to point out that in the most part of the above mentioned papers, which involves the Laplacian operator, the radial functions space plays an important role because of its compactness embedding properties. At contrary, for problems involving the  Kadomtsev-Petviashvili equation,  we do not have a similar result for radial functions any longer and therefore we need to use different arguments to prove our main results. Here, we will adapt for our problem the variational approach explored in  Jeanjean \cite{jeanjean} and  Hirata,  Ikoma and  Tanaka \cite{HIT} (see also \cite{ADP,CPDS})  by considering  an auxiliary functional that allows to construct a suitable Palais-Smale sequence, which {\em almost} satisfies a Pohozaev type identity, see Sections 3 and 4 for more details.   

The paper is organized as follows. In Section 2 we present our functional setting describing its embeddings properties. In the last two sections, instead, we treat the positive mass case and the zero mass case, proving our main existence results.

\vspace{0.5 cm}

\noindent {\bf Notations:} \, Throughout the paper, unless explicitly stated, the symbol $C$ will always denote a generic positive constant,
which may vary from line to line. The symbols \lq\lq $\rightarrow$ \rq\rq and \lq\lq $\rightharpoonup$ \rq\rq denote, respectively, strong and weak convergence, and  all the convergences involving sequences in $n\in \N$ are as $n\rightarrow \infty.$ Moreover we denote by  $|\cdot |_q$ the usual norm of Lebesgue space $L^q(\R^N)$, for $q\in [1,+\infty]$. Finally, for all $(x,y)\in \R \times \R^{N-1}$ and $r>0$, we denote by $B_r((x,y))$ the ball centred in $(x,y)$ with radius $r$. 

\section{Functional setting }

We intend to study our problem using variational methods and, as first step, we introduce our functional setting.

\begin{deff} On $Y=\{ g_x: g \in C_{0}^{\infty}(\R^N)\}$ define the inner product  
\begin{equation*}
(u,v)=\int_{\R^N} \left( u_x v_x +D^{-1}_{x}\nabla_y  u \cdot D^{-1}_{x} \nabla_y v + uv\right)  dV,
\end{equation*}
with the corresponding norm 
\begin{equation*}
\|u\|=\left(\int_{\R^N} \left( |u_x|^2 +|D^{-1}_{x}\nabla_y u|^2  + |u|^2\right)  dV \right)^{\frac{1}{2}},
\end{equation*}
where $\nabla_y=(\frac{\partial}{\partial y_1}, \ldots,\frac{\partial}{\partial y_{N-1}} ) $ and $dV= dx \ dy.$

We say that $u:\R^N \rightarrow \R$ belongs to $X$  if there exists  a sequence $\{u_n\}\subset  Y$ such that
\begin{description}
\item[a)] $ u_n \rightarrow u$ a.e. on $\R^N$,
\item[b)] $\|u_j - u_k\|\rightarrow 0$, as $ j, k \rightarrow \infty.$
\end{description}
\end{deff}

\begin{deff} On $Y=\{ g_x: g \in C_{0}^{\infty}(\R^N)\}$ define the inner product  
\begin{equation*}
(u,v)_0=\int_{\R^N} \left( u_x v_x +D^{-1}_{x}\nabla_y  u \cdot D^{-1}_{x} \nabla_y v \right)  dV
\end{equation*}
with the corresponding norm 
\begin{equation*}
\|u\|_0=\left(\int_{\R^N} \left( |u_x|^2 +|D^{-1}_{x}\nabla_yu|^2  \right)  dV \right)^{\frac{1}{2}}.
\end{equation*}
We say that $u:\R^N \rightarrow \R$ belongs to $X_0$  if there exists  a sequence $\{u_n\}\subset  Y$ such that
\begin{description}
\item[a)] $ u_n \rightarrow u$ a.e. on $\R^N$,
\item[b)] $\|u_j - u_k\|_0\rightarrow 0$, as  $j, k \rightarrow \infty.$
\end{description}
\end{deff}

From definition of $X$ and $X_0$, the embedding $(X,\|\,\,\|) \hookrightarrow (X_0,\|\,\,\|_0)$ is continuous.  

The spaces $X$ and $X_0$ endowed with inner products and norms given above are Hilbert spaces. Moreover, we have the following continuous embeddings, whose proof can be found in \cite[Theorem 15.7  p. 323]{Besov}, also in \cite[Lemma 2.1]{Liang} and \cite[Lemma 2.3]{Xuan},
\begin{equation} \label{continuous}  
X \hookrightarrow L^{q}(\R^N), \quad\mbox{for} \ 1\leq q \leq \bar{N}.
\end{equation}
Regarding to compact embeddings, De  Bouard and Saut in \cite[Remark 1.1]{SautB1} for $N=2,3,$ and Xuan in \cite[Lemma 2.4]{Xuan} for higher dimensions, they have proved that the embeddings below 

\begin{equation}\label{compact}  
X \hookrightarrow L^{q}_{loc}(\R^N), \quad \mbox{for} \ 1\leq q < \bar{N},
\end{equation}
are compacts. 

The following lemma is a Lions' type result, see \cite{Lions}, for the space $X$, whose the proof can be found in 
\cite[Lemma 2.5]{Xuan}.

\begin{lem} \label{Lions} (\cite[Lemma 2.5]{Xuan})\,If $\{u_n\}$ is a sequence bounded in $X$ and if
	$$
	\sup_{(x,y)\in \mathbb{R}^N}\int_{B_r((x,y))}|u_n|^{2}\,dV \to 0,\ \mbox{as}\ n \rightarrow \infty,
	$$
	then $u_n \to 0$ in $L^{q}(\mathbb{R}^N)$ for all $q \in (2,\bar{N})$.
\end{lem}

With respect to the continuous embedding $X_0 \hookrightarrow L^{\bar{N}}(\mathbb{R}^N)$, we have the following result
\begin{lem}\label{le:l6} (\cite[Lemma 2.3]{Xuan}) There exists a constant $S>0$ such that
\begin{equation*}
|u|_{6} \leq S \left(\int_{\mathbb{R}^N}(|u_x|^2 +|D^{-1}_{x}\nabla_y u|^2)\,dV \right)^{\frac{1}{2}}, \quad \forall u \in X_0.
\end{equation*}	
\end{lem} 

Finally, before concluding this section, we would like to point out that the same approach explored in \cite[Lemma 2.4]{Xuan}, or \cite[Theorem 7.3]{Willem}, gives  
\begin{equation}\label{compact2}  
X_0 \hookrightarrow L^{q}_{loc}(\R^N), \quad \mbox{for} \quad 1\leq q < \bar{N},
\end{equation}
are compact. Moreover, it is very important to say that if $\Omega \subset \R^N$ is a smooth bounded domain and $q \in [1,\bar{N}]$, there exists $C>0$ such that
\begin{equation}\label{Estimativa 2} 
|u|_{L^{q}(\Omega)} \leq C \left(\int_{\Omega} \left( |u_x|^2 +|D^{-1}_{x}\nabla_y u|^2  + |u|^2\right)  dV \right)^{\frac{1}{2}}, \quad \forall u \in X_0.
\end{equation}
The above information follows from  properties involving anisotropic Sobolev spaces, for more details see  Besov, Il'in, and  Nikolskiı \cite[Chapter 3]{Besov}.

 \section{The existence of solution for positive mass case}

Through this section we will assume \eqref{hp}-\eqref{h>0}. 
We will find solutions of equation (\ref{eq4}) as critical points of  the energy functional $I:X \longrightarrow \R$ given by
$$
I(u)=\frac{1}{2}\|u\|^2-\int_{\R^N}H(u)\, dV.
$$

\begin{lem} \label{PM1} The functional $I:X \to \mathbb{R}$ verifies the mountain pass geometry, that is, 
\begin{itemize}
\item[(i)] there are $\alpha, \rho>0$ such that
$I(u) \geq \alpha$, for $\|u\|=\rho$;	
\item[(ii)] there is $e \in X \setminus\{0\}$ such that $I(e)<0$, with $\|e \|>\rho$. 
\end{itemize}
\end{lem}
\begin{proof} 
The proof of $(i)$ follows by using standard arguments involving the growth condition on $h$, then it will be omitted. In order to prove $(ii)$, from \eqref{h>0}  there is $\phi \in  C^{\infty}_0(\mathbb{R}^{N})$ such that
$$
\int_{\mathbb{R}^N}\left(H(\phi)-\frac{\phi^2}{2}\right)dV>0.
$$		
For $t>0$, setting 
$$
w_t(x,y)=\phi({x}/{t},{y}/{t^2}), \qquad \hbox{for }(x,y)\in \R\times\R^{N-1},
$$
by simple calculations, we derive
$$
I_\lambda(w_t)=\frac{t^{2N-3}}{2}\|\phi\|_{0}^2-t^{2N-1}\int_{\mathbb{R}^N}\left(H(\phi)-\frac{\phi^2}{2}\right)dV \to -\infty \quad \mbox{as} \ t \to +\infty.
$$	
Therefore, $(ii)$ follows by choosing $e=w_t$ with $t$ large enough.	
\end{proof}

We set
$$
\Gamma=\{ \gamma \in C([0,1],X)\,:\,\gamma(0) =0,\  \gamma(1)=e \}
$$  	
and
$$
\sigma=\inf_{\gamma \in \Gamma}\max_{t \in [0,1]}I(\gamma(t)).
$$
Clearly, by Lemma \ref{PM1}, $\sigma\ge \alpha>0$.

Following \cite{HIT,jj}, we introduce an auxiliary functional $\tilde I\in C^1(\R\times X,\R)$ given by
\[
\tilde I(\t,u)=\frac{e^{(2N-3)\t}}{2}\|u\|^2_0+\frac{e^{(2N-1)\t}}{2}|u|_2^2-e^{(2N-1)\t}\ird H(u) \, dV.
\]
The following properties hold, for all $(\t,u)\in \R\times X$,
\begin{align*}
\tilde I(0,u)&=I(u),
\\
\tilde I (\t,u)&=I(u(e^{-\t}x,e^{-2\t}y)). 
\end{align*}
We equip a standard product norm $\|(\t,u)\|_{\R\times X}=(|\t|^2+\|u\|^2)^{1/2}$ to $\R\times X$.

By Lemma \ref{PM1}, it is easy to see that also the functional $\tilde I$ satisfies the mountain pass geometry. More precisely, the following holds
\begin{lem} \label{PM1-tilde} The functional $\tilde I:\R\times X \to \mathbb{R}$ verifies the mountain pass geometry, that is, 
\begin{itemize}
\item[(i)] there are $\alpha, \rho>0$ such that
$\tilde I(\t,u) \geq \alpha$, for $\|(\t,u)\|_{\R\times X}=\rho$;	
\item[(ii)] there is $\tilde e \in \R\times X \setminus\{0\}$ such that $\tilde I(\tilde e)<0$, with $\|\tilde e \|_{\R\times X}>\rho$. 
\end{itemize}
\end{lem}

\begin{proof}
For $(ii)$ it is sufficient to take $\tilde e=(0,e)$,  while for $(i)$ just follows by Lemma \ref{PM1}.
\end{proof}

In what follows, we define the mountain pass level $\tilde \sigma$ for $\tilde I$ by
 $$
\tilde \sigma=\inf_{\tilde\gamma \in \tilde \Gamma}\max_{t \in [0,1]}\tilde I(\tilde \gamma(t)),
$$
where
$$
\tilde \Gamma=\{ \tilde \gamma \in C([0,1],\mathbb{R}\times X)\,:\,\tilde \gamma(0) =(0,0),\  \tilde\gamma(1)=(0,e) \}.
$$  	
Hence,  $\tilde \sigma\ge \alpha>0$. Arguing as in \cite[Lemma 4.1]{HIT},  we derive
\begin{lem}
The mountain pass levels of $I$ and $\tilde I$ coincide, namely $\sigma=\tilde \sigma$.
\end{lem}

Now, as a immediate consequence of Ekeland's variational principle, we have the result below whose proof follows as in \cite[Lemma 2.3]{jeanjean}.
\begin{lem}\label{le:ekeland}
Let $\eps>0$. Suppose that $\tilde \gamma \in \tilde{\Gamma}$ satisfies 
\[
\max_{t \in [0,1]}\tilde I(\tilde \gamma(t))\le \sigma+\eps,
\]
then there exists $(\t, u)\in \R\times X$ such that
\begin{enumerate}
\item ${\rm dist}_{\R \times X}\big((\t,u),\tilde{\gamma}([0,1])\big)\le 2 \sqrt{\eps}$;
\item $\tilde{I}(\t,u)\in [\sigma-\eps,\sigma+\eps]$;
\item $\|D\tilde{I}(\t,u)\|_{\R \times X^*}\le 2 \sqrt{\eps}$.
\end{enumerate}
\end{lem}

Arguing as in \cite{HIT}, by Lemma \ref{le:ekeland}, the following proposition holds
\begin{prop}\label{pr:sequence}
There exists a sequence $\{(\t_n,u_n)\} \subset \R \times X$ such that, as $n \to +\infty$, we get 
\begin{enumerate}
\item $\t_n \to 0$;
\item $\tilde{I}(\t_n,u_n)\to \sigma$; 
\item $\de_\t\tilde{I}(\t_n,u_n)\to 0$; 
\item $\de_u\tilde{I}(\t_n,u_n)\to 0$ strongly in $X^*$. 
\end{enumerate}

\end{prop}

After the above study, we are ready to prove Theorem \ref{T1}.

\begin{proof}[Proof of Theorem \ref{T1}]
By Proposition \ref{pr:sequence}, there exists a sequence $\{(\t_n,u_n)\} \subset \R \times X$ such that
\begin{equation}\label{sistema}
\begin{cases}
\dis\frac{e^{(2N-3)\t_n}}{2}\|u_n\|^2_0+\frac{e^{(2N-1)\t_n}}{2}|u_n|_2^2-e^{(2N-1)\t_n}\ird H(u_n) \, dV=\sigma+o_n(1),
\\[4mm]
\dis\frac{e^{(2N-3)\t_n}}{2}\|u_n\|^2_0+\frac{(2N-1)e^{(2N-1)\t_n}}{2}|u_n|_2^2-(2N-1)e^{(2N-1)\t_n}\ird H(u_n) \, dV=o_n(1),
\\[4mm]
\dis e^{(2N-3)\t_n}\|u_n\|^2_0+e^{(2N-1)\t_n}|u_n|_2^2-e^{(2N-1)\t_n}\ird h(u_n)u_n \, dV=o_n(1)\|u_n\|.
\end{cases}
\end{equation}
From the first and the second equation of the previous system we get
\[
(N-1)e^{\t_n}\|u_n\|^2_0=(2N-1)\sigma+o_n(1),
\]
and so, since $\t_n \to 0$, we infer that $\{u_n\}$ is bounded in $X_0$ and so also in $L^{\bar{N}}(\RD)$, by Lemma \ref{le:l6}. Observe that, by \eqref{hp} and \eqref{h'}, we deduce that for any $\delta>0$, there exists $C_\delta>0$ such that
\begin{equation*}
|h(t)|\le \delta |t|+C_\delta |t|^{\bar{N}-1}, \qquad\hbox{for all }t\in \R.
\end{equation*}
Hence, by the third equation of \eqref{sistema}, using again the fact that  $\t_n \to 0$, we find
\[
\|u_n\|^2\le C e^{(2N-1)\t_n}\ird h(u_n)u_n \, dV+o_n(1)\|u_n\|
\le C\left( \delta|u_n|^2_2 +C_\delta |u_n|^{\bar{N}}_{\bar{N}}\right)+o_n(1)\|u_n\|.
\]
Then for $\delta$ small enough, and using the fact that $\{|u_n|_{\bar{N}}\}$ is bounded, it follows that
$$
\|u_n\|^2\le C, \quad \forall n \in \mathbb{N},
$$
for some $C>0$, showing that $\{u_n\}$ is actually bounded in $X$. Moreover, by the continuous embedding \eqref{continuous}, we also have for all $p<\bar{N}-1,$
\[
|u_n|_{p+1}^{2}\le C \|u_n\|^{2}\le C|u_n|^{p+1}_{p+1}, \quad \forall n \in \mathbb{N},
\]
and so there exists $c>0$ such that $|u_n|_{p+1}\ge c>0$ for all $n \in \mathbb{N}$. Hence, by Lemma \ref{Lions}, there exist a sequence of points $\{(x_n,y_n)\}\subset\RD$ and $r,\beta>0$ such that 
$$
\int_{B_r((x_n,y_n))}|u_n|^{2}\,dV \ge \beta> 0.
$$
Hence, calling $v_n=u(\cdot-x_n,\cdot-y_n)$, being $\{v_n\}$ a bounded sequence in $X$, up to a subsequence, we must have 
\[
v_n \rightharpoonup v\neq 0 \qquad \hbox{weakly in }X.
\]
By the invariance by translations of $\tilde{I}$, we have that $\de_u\tilde{I}(\t_n,v_n)\to 0$ strongly in $X^*$, and so, since $\t_n \to 0$ and by the local compact embedding \eqref{compact}, we conclude that $I'(v)=0$, thus $v$ is a non-trivial solution of \eqref{eq4}.

\end{proof}

 \section{The existence of solution for the zero mass case}

Through this section we will assume \eqref{h>0}-\eqref{hf5}. 
We start with a technical lemma, which will be used later on. 
\begin{lem}\label{le:supK} 
Let $\{w_n\} \subset X_0$ be a bounded sequence in $X_0$ with  
$$
\lim_{n \to +\infty}\sup_{(x,y)\in \mathbb{R}^N} \int_{K(x,y)}|f(w_n)w_n|\, dV=0,
$$	
where $K(x,y)=(x-1,x+1) \times B_1(y)$. Then $\displaystyle \lim_{n \to +\infty}\int_{\mathbb{R}^N}f(w_n)w_n\, dV=0$. 
\end{lem}
\begin{proof} By \eqref{hf5}, there is $C>0$ such that
$$
|f(t)t| \leq C|t|^{\bar{N}}, \quad \forall t \in \mathbb{R}. 
$$ 
The above inequality combines with (\ref{Estimativa 2})  to give  
$$
\int_{K(x,y)}\!\!|f(w_n)w_n|\, dV \leq C\int_{K(x,y)}\!\!|w_n|^{\bar{N}}\, dV \leq C_1\left[\int_{K(x,y)}\!\!\left(|(w_n)_x|^2+|D_x^{-1} \nabla_y w_n|^{2}+|w_n|^2 \right)\, dV\right]^{\bar{N}/2}.
$$
Thus, for all $\lambda \in (0,1)$,
$$
\int_{K(x,y)}\!\!\!|f(w_n)w_n| dV \leq C_1^{\lambda}\left[\int_{K(x,y)}\!\!\!\left(|(w_n)_x|^2\!+|D_x^{-1}\nabla_y  w_n|^{2}\!+|w_n|^2\right)dV\right]^{\lambda\bar{N}/2} \!\left(\int_{K(x,y)}\!\!\!|f(w_n)w_n|dV\right)^{1-\lambda}.
$$
Setting $\lambda=2/\bar{N}$, we get
$$
\int_{K(x,y)}\!\!\!|f(w_n)w_n| dV \leq C_1^{2/\bar{N}}\left[\int_{K(x,y)}\!\!\!\left(|(w_n)_x|^2\!+|D_x^{-1}\nabla_y w_n|^{2}\!+|w_n|^2\right) dV\right] \!\left(\int_{K(x,y)}\!\!\!|f(w_n)w_n| dV\right)^{1- 2/\bar{N}}.
$$
By using the fact that 
$$
|w_n|_{L^{2}(K(x,y))} \leq C_*|w_n|_{L^{\bar{N}}(K(x,y))}, \quad \forall n \in \mathbb{N},
$$ 
for some constant $C_*>0$ independent of $(x,y) \in \R^N$, we get 
$$
\int_{\mathbb{R}^N}|f(w_n)w_n|\, dV \leq C_1^{2/\bar{N}}\|w_n\|_0^2 \left(\sup_{(x,y)\in \mathbb{R}^N}\int_{K(x,y)}|f(w_n)w_n|\, dV\right)^{1-2/\bar{N}}
$$
and so,
$$
\int_{\mathbb{R}^N}|f(w_n)w_n|\, dV \leq C_2\left(\sup_{(x,y)\in \mathbb{R}^N}\int_{K(x,y)}|f(w_n)w_n|\, dV \right)^{1- 2/\bar{N}},
$$
for all $n \in \mathbb{N}$ and for some $C_2>0$. From where it follows that
$$
\lim_{n\to +\infty}\int_{\mathbb{R}^N}|f(w_n)w_n|\, dV =0
$$
and so the claim.
\end{proof}

Associated with equation (\ref{eq4'}), by \eqref{hf}, we have the energy functional $I_0:X_0 \longrightarrow \R$ given by
$$
I_0(u)=\frac{1}{2}\|u\|_0^2-\int_{\R^N}F(u)\, dV.
$$

\begin{lem} \label{PM10} The functional $I_0:X_0 \to \mathbb{R}$ verifies the mountain pass geometry, that is, 
\begin{itemize}
\item[(i)] there are $\alpha, \rho>0$ such that
$I_0(u) \geq \alpha$, for $\|u\|_0=\rho$;	
\item[(ii)] there is $e \in X_0 \setminus\{0\}$ such that $I_0(e)<0$, with $\|e \|_0>\rho$. 
\end{itemize}
\end{lem}

\begin{proof}
The proof is similar as in Lemma \ref{PM1}, using \eqref{hf5} and \eqref{f>0}.
\end{proof}

We set
$$
\Gamma_0=\{ \gamma \in C([0,1],X_0)\,:\,\gamma(0) =0,\  \gamma(1)=e \}
$$  	
and
$$
\sigma_0=\inf_{\gamma \in \Gamma_0}\max_{t \in [0,1]}I_0(\gamma(t)).
$$
Clearly, by Lemma \ref{PM10}, $\sigma_0\ge \alpha>0$.

As in the previous section, we introduce the auxiliary functional $\tilde I_0\in C^1(\R\times X_0,\R)$ given by
\[
\tilde I_0(\t,u)=\frac{e^{(2N-3)\t}}{2}\|u\|^2_0-e^{(2N-1)\t}\ird F(u) \, dV.
\]
The following properties hold, for all $(\t,u)\in \R\times X_0$,
\begin{align*}
\tilde I_0(0,u)&=I_0(u),
\\
\tilde I_0 (\t,u)&=I_0(u(e^{-\t}x,e^{-2\t}y)). 
\end{align*}
We equip $\R\times X_0$ with the standard product norm $\|(\t,u)\|_{\R\times X_0}=(|\t|^2+\|u\|_0^2)^{1/2}$. Arguing as in Lemma \ref{PM1-tilde}, we can see that $\tilde I_0$ satisfies the mountain pass geometry. More precisely, we have
\begin{lem} \label{PM10-tilde} The functional $\tilde I_0:\R\times X_0 \to \mathbb{R}$ verifies the mountain pass geometry, that is, 
\begin{itemize}
\item[(i)] there are $\alpha, \rho>0$ such that
$\tilde I_0(\t,u) \geq \alpha$, for $\|(\t,u)\|_{\R\times X_0}=\rho$;	
\item[(ii)] there is $\tilde e \in \R\times X_0 \setminus\{0\}$ such that $\tilde I_0(\tilde e)<0$ with $\|\tilde e \|_{\R\times X_0}>\rho$. 
\end{itemize}
\end{lem}

%

Hence we define the mountain pass level $\tilde \sigma_0$ for $\tilde I_0$ by
 $$
\tilde \sigma_0=\inf_{\tilde\gamma \in \tilde \Gamma_0}\max_{t \in [0,1]}\tilde I_0(\tilde \gamma(t)),
$$
where
$$
\tilde \Gamma=\{ \tilde \gamma \in C([0,1],\mathbb{R}\times X_0)\,:\,\tilde \gamma(0) =(0,0),\  \tilde\gamma(1)=(0,e) \}.
$$  	
From definition $\sigma_0$, we have $\tilde \sigma_0\ge \alpha>0$, $\sigma_0=\tilde \sigma_0$ and the following proposition 

\begin{prop}\label{pr:sequence0}
There exists a sequence $\{(\t_n,u_n)\} \subset \R \times X_0$ such that, as $n \to +\infty$, we get 
\begin{enumerate}
\item $\t_n \to 0$;
\item $\tilde{I_0}(\t_n,u_n)\to \sigma$; 
\item $\de_\t\tilde{I_0}(\t_n,u_n)\to 0$; 
\item $\de_u\tilde{I_0}(\t_n,u_n)\to 0$ strongly in $X_0^*$. 
\end{enumerate}

\end{prop}

Now, we are going to prove Theorem \ref{T2}.

\begin{proof}[Proof of Theorem \ref{T2}]
By Proposition \ref{pr:sequence0}, there exists a sequence $\{(\t_n,u_n)\} \subset \R \times X_0$ such that, as $n \to +\infty$, we have
\begin{equation}\label{sistema0}
\begin{cases}
\dis\frac{e^{(2N-3)\t_n}}{2}\|u_n\|^2_0-e^{(2N-1)\t_n}\ird F(u_n) \, dV=\sigma_0+o_n(1),
\\[4mm]
\dis\frac{e^{(2N-3)\t_n}}{2}\|u_n\|^2_0-(2N-1)e^{(2N-1)\t_n}\ird F(u_n) \, dV=o_n(1),
\\[4mm]
\dis e^{(2N-3)\t_n}\|u_n\|^2_0-e^{(2N-1)\t_n}\ird f(u_n)u_n \, dV=o_n(1)\|u_n\|_0.
\end{cases}
\end{equation}
From the first and the second equation of the previous system we get
\begin{equation}\label{3sigma0}
(N-1)e^{\t_n}\|u_n\|^2_0=(2N-1)\sigma_0+o_n(1),
\end{equation}
and so, since $\t_n \to 0$, we infer that $\{u_n\}$ is bounded in $X_0$ and so also in $L^{\bar{N}}(\RD)$, by Lemma \ref{le:l6}.
Hence, by the third equation of \eqref{sistema0} and \eqref{3sigma0}, using again the fact that  $\t_n \to 0$, there exists $c>0$ such that 
\begin{equation*}
\ird f(u_n)u_n \, dV\ge c>0, \quad \forall n \in \mathbb{N}.
\end{equation*}
Hence, by Lemma \ref{le:supK}, there exist a sequence of points $\{(x_n,y_n)\}\subset\RD$ and $\bar c>0$ such that 
$$
\int_{K(x_n,y_n)}|f(u_n)u_n|\,dV \ge \bar c> 0, \quad \forall n \in \mathbb{N}.
$$
Hence, calling $v_n=u(\cdot-x_n,\cdot-y_n)$, and being $\{v_n\}$ a bounded sequence in $X_0$, up to a subsequence, we have 
\[
v_n \rightharpoonup v \qquad \hbox{weakly in }X_0.
\]
By \eqref{hf5}, there is $C>0$ such that
$$
|f(t)t| \leq \frac{\bar c}{2M}|t|^{\bar{N}}+C|t|^{2}, \quad \forall t \in \mathbb{R},
$$  
where $M=\sup_{n \in \mathbb{N}} \int_{\mathbb{R}^N}|v_n|^{\bar{N}}\,dV$. From this, 
$$
C\int_{K(0,0)}|v_n|^{2}\,dV \geq \frac{\bar c}{2}, \quad \forall n \in \mathbb{N}
$$ 
and so, by (\ref{compact2}), 
$$
\int_{K(0,0)}|v|^{2}\,dV \geq \frac{\bar c}{2C}, \quad \forall n \in \mathbb{N},
$$ 
implying that $v \neq 0$. 
\\
Now, we will prove that $v$ is a solution for (\ref{eq4'}). To this end, without loss of generality, we can assume that
$$
v_n(x,y) \to v(x,y) \quad \mbox{a.e. in} \ \mathbb{R}^N,
$$
and so, by continuity, 
$$
f(v_n(x,y)) \to f(v(x,y)) \quad \mbox{a.e. in} \ \mathbb{R}^N.
$$
By \eqref{hf5},  $\{f(v_n)\}$ is a bounded sequence in $L^{\frac{\bar{N}}{\bar{N}-1}}(\mathbb{R}^N)$, and so, there exists $g\in L^{\frac{\bar{N}}{\bar{N}-1}}(\mathbb{R}^N)$ such that
$f(v_n) \rightharpoonup g$ in  $L^{\frac{\bar{N}}{\bar{N}-1}}(\mathbb{R}^N)$. 
It is standard to prove that $g=f(v)$. These informations yield, in particular, that
$$
\int_{\mathbb{R}^N}f(v_n) \phi \, dV \to \int_{\mathbb{R}^N}f(v) \phi \, dV,\quad \forall \phi \in X_0.
$$
This limit combines with $\de_u\tilde{I}_0(\t_n,v_n) [\phi] \to 0$ to give $I'_0(v)\phi=0$, for any $\phi \in X_0$, showing that $v$ is a nontrivial for (\ref{eq4'}).

\end{proof}

\end{document}